\documentclass[12pt,reqno]{amsart}

\usepackage{graphicx}
\usepackage{tikz}
\usepgflibrary{shapes}
\usepackage{epsfig}
\usepackage{amsfonts,amsmath,amsthm}
\usepackage{amssymb}

\usepackage{enumitem}
\usepackage{lipsum}

\usepackage{caption}
\usepackage{colortbl} 
\usepackage{caption}
\usepackage[makeroom]{cancel}
\usepackage{tikz}
\newcommand{\commentout}[1]{}

\def \Rset {{\mathbb R}}

\def \Zset {{\mathbb Z}}

\def \Nset {{\mathbb N}}

\newcommand{\be}{\begin{equation}}
\newcommand{\ee}{\end{equation}}
\newcommand{\ba}{\begin{eqnarray}}
\newcommand{\ea}{\end{eqnarray}}
\newcommand{\bi}{\begin{itemize}}
\newcommand{\ei}{\end{itemize}}
\newcommand{\br}{\begin{eqnarray}}
\newcommand{\er}{\end{eqnarray}}

\newcommand{\Div}[1]{\text{div}\left(#1\right)}

\newtheorem{theo}{Theorem}[section]
\newtheorem{defin}{Definition}[section]

\newtheorem{lem}{Lemma}[section]
\newtheorem{cor}{Corollary}[section]
\newtheorem{rmk}{Remark}[section]

\numberwithin{equation}{section}

\begin{document}
\title[Bifurcation of homogenization and nonhomogenization]{Bifurcation of homogenization and nonhomogenization of  the curvature G-equation with shear flows}

\author{Hiroyoshi Mitake}
\address[H. Mitake]{
	Graduate School of Mathematical Sciences, 
	University of Tokyo 
	3-8-1 Komaba, Meguro-ku, Tokyo, 153-8914, Japan}
\email{mitake@g.ecc.u-tokyo.ac.jp}

\author{Connor Mooney}
\address[C. Mooney]{
Department of Mathematics, 
University of California at Irvine, 
California 92697, USA}
\email{mooneycr@math.uci.edu}

\author{Hung V. Tran}
\address[H. V. Tran]
{
Department of Mathematics, 
University of Wisconsin Madison, Van Vleck Hall, 480 Lincoln Drive, Madison, Wisconsin 53706, USA}
\email{hung@math.wisc.edu}

\author{Jack Xin}
\address[J. Xin]
{
Department of Mathematics, 
University of California at Irvine, 
California 92697, USA}
\email{jack.xin@uci.edu}

\author{Yifeng Yu}
\address[Y. Yu]
{
Department of Mathematics, 
University of California at Irvine, 
California 92697, USA}
\email{yifengy@uci.edu}

\thanks{HM was partially supported by the JSPS grants: KAKENHI \#22K03382, \#21H04431, \#20H01816, \#19K03580, \#19H00639.
CM was partly supported by Alfred P. Sloan Research Fellowship, NSF CAREER grant DMS-2143668.
HT was supported in part by NSF CAREER grant DMS-1843320 and a Vilas Faculty Early-Career Investigator Award.
JX was partly supported by NSF grants DMS-1952644, DMS-2309520 .
 YY was partly supported by NSF grant DMS-2000191.}

\date{}


\begin{abstract}
The level-set curvature G-equation, a well-known model in turbulent combustion, has the following form
$$
G_t + \left(1-d\, \Div{\frac{DG}{|DG|}}\right)_+|DG|+V(X)\cdot DG=0.
$$
Here the cutoff  correction $()_+$ is imposed to  avoid non-physical negative local burning velocity.   
The existence of the effective burning velocity has been established for a large class of physically relevant incompressible flows $V$ in two dimensions \cite{GLXY} via game theory dynamics.    
In this paper, we show that the effective burning velocity associated with shear flows in dimensions three or higher ceases to exist when the flow intensity surpasses a bifurcation point.  The characterization of the bifurcation point  in three dimensions is closely related to the regularity theory of two-dimensional minimal surface type equations due to \cite{S1}.  
As a consequence,  a bifurcation also exists for  the validity of  full homogenization of  the curvature G-equation associated with shear flows.

\end{abstract}

\maketitle

\section{Introduction}

\subsection{Overview}
The level-set curvature G-equation is a well-known model in turbulent combustion \cite{P2000, W1985}:
\be
G_t + \left(1- d\,  {\rm div}\left({DG\over |DG|}\right)\right)_{+}|DG|+V(X)\cdot DG=0.\label{ge1}
\ee
Solutions to \eqref{ge1} are interpreted in the viscosity sense (see \cite{CL1992, GLXY, G} for instance).
The zero level set  $\{G(X,t)=0\}$ represents the flame front at time $t$. 
The burnt and unburnt regions are $\{G(X,t) < 0\}$ and $\{G(X,t) > 0\}$, respectively. 
See the left picture in  Figure 1 below.  
Here, $d>0$ is the Markstein number,   $X\in \Rset^{n+1}$, and $V:\Rset^{n+1}\to \Rset^{n+1}$ is  the velocity of the ambient fluid.  
The local burning velocity  (laminar flame speed) is given by 
$$
s_l=(1- d\; \kappa)_{+},
$$
where the mean curvature term $\kappa= {\rm div}\left({DG\over |DG|}\right)$ was first introduced in \cite{Mar} to approximate the dependence of local burning velocity on the variance of temperature along the flame front.  
Intuitively, if the flame front bends  toward the cold region (unburned area, point C in the right picture of Figure 1 below),  the flame propagation slows down. 
If the flame front bends toward the hot spot (burned area, point B in right picture of Figure 1 below), it burns faster. 
The positive  cutoff part $(a)_{+}=\max\{a,0\}$ for $a\in \Rset$ is imposed to avoid negative laminar flame speed since materials 
cannot be ``unburned".   
To maintain physical validity, it is necessary to add the cutoff  correction since,  mathematically,  solutions of \eqref{ge1} might develop large positive curvature as time evolves. Meanwhile,   the cutoff makes the equation more degenerate.  See \cite{ZR2007} and  \cite[Section 6.4]{OF2002} for numerical computations of the curvature G-equations (with or without the cutoff, respectively).

\begin{center}\label{curv3}
\includegraphics[scale=0.45]{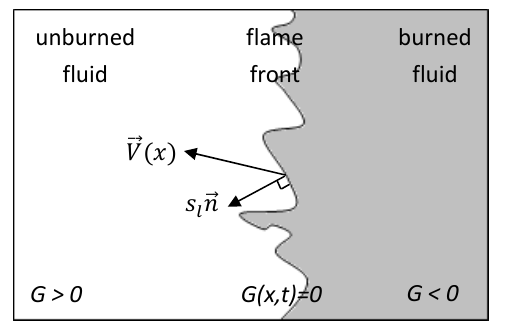} \hspace{3cm}
\includegraphics[scale=0.45]{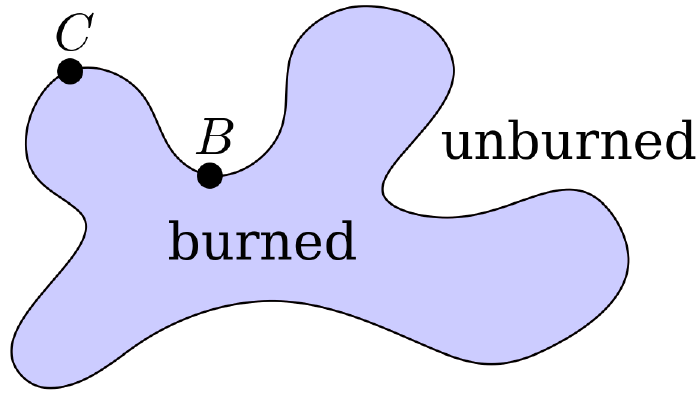}
\captionof{figure}{Left. Level-set formulation \qquad \qquad  Right. Curvature effect}
\end{center}

A significant problem in the combustion literature is to understand how the curvature term impacts the prediction of the burning velocity \cite{Ron}. 
To study this mathematically, the first step is to establish the existence of the effective burning velocity by averaging the geometry and small scales in the flow.  
More precisely,  given a direction $P\in \Rset^{n+1}$, let $G(X,t)\in C(\Rset^{n+1}\times [0,\infty))$ be the solution to the following equation
\be
\begin{cases}
G_t + \left(1- d\,  {\rm div}\left({DG\over |DG|}\right)\right)_{+}|DG|+V(X)\cdot DG=0,\label{ge3}\\[3mm]
G(X,0)=P\cdot X. 
\end{cases}
\ee
The goal then is to determine whether  there exists a constant $\overline H_{+}(P,d)$ such that
\be\label{time-limit}
-\lim_{t\to +\infty}{G(X,t)-P\cdot X\over t}=\overline H_{+}(P,d) \quad \text{for all $X\in \Rset^{n+1}$}.
\ee
The constant limit $\overline H_+(P,d)$,  if exists, can be viewed as the effective burning velocity (or called the turbulent burning velocity in combustion literature) along direction $P$ under the curvature G-equation model.  

The existence of   \eqref{time-limit}  for all $P\in \Rset^{n+1}$ implies the  strong homogenization of the curvature G-equation in our setting.  See Theorem \ref{theo:hom} for the precise statement. On the other hand,  if \eqref{time-limit} does not hold, then the  strong homogenization for \eqref{ge3} fails.  
Precisely speaking,   suppose that there exist $X_1, X_2\in \Rset^{n+1}$ such that 
 $$
\liminf_{t\to +\infty}{G(X_1,t)-P\cdot X_1 \over t}<\limsup_{t\to +\infty}{G(X_2,t)-P\cdot X_2 \over t},
$$ 
then by the $\Zset^{n+1}$ periodicity of $X\mapsto G(X,t)-P\cdot X$, it is easy to show that for any $(X,t)\in \Rset^{n+1}\times (0,\infty)$, 
$$
\liminf_{(Y,s,\epsilon)\to (X,t,0)}G_\epsilon (Y,s)<\limsup_{(Y,s,\epsilon)\to (X,t,0)}G_\epsilon (Y,s)
$$
for $G_{\epsilon}(X,t)=\epsilon G({X\over \epsilon}, {t\over \epsilon})$, which satisfies equation \eqref{ge1E} with $G_\epsilon(X,0)=P\cdot X$. 

 For the curvature G-equation, one of  the main difficulties to establish \eqref{time-limit}   is the lack of coercivity with respect to the gradient variable. 
 The existence of $\overline H_+(P,d)$ has been established in \cite{GLXY}  via the Lagrangian approach for a large class of physically meaningful two dimensional incompressible  periodic flows $V$.  
 One of  the most interesting examples is  the cellular flow appearing frequently in mathematics and physics literature \cite{CG}, e.g.,  for $X=(x_1,x_2)$,
  $$
 V(X)=(-H_{x_2}, H_{x_1}) \quad \text{ with $H(X)=A\sin x_1\sin x_2$}.
 $$
 Here, the positive parameter $A$ stands for the flow intensity. 
 
 The proof  in \cite{GLXY}  relies on the analysis of a game theory interpretation of curvature type operators in two dimensions  \cite{KS1}.  
 It remains an interesting question whether the existence also holds for three dimensional incompressible flows.  
 Surprisingly, this is not the case even for shear flows, a simple class of incompressible flows appearing in apparatus such as slot burners.  Meanwhile, the limit \eqref{time-limit} is easy to establish for two dimensional shear flows (see Remark \ref{rmk:1d}).

\subsection{Assumptions and main results}
We only focus on the shear flow setting in dimensions three or higher in this paper. 
For $n\geq 2$,  let $f:\Rset^n\to \Rset$  be a Lipschitz continuous $\Zset^n$-periodic function satisfying 
\be\label{c1}
 \{x\in \Rset^n\,:\, f(x)=\max_{\Rset^n}f\}=\Zset^n \quad \text{and} \quad \{x\in \Rset^n\,:\, f(x)=\min_{\Rset^n}f\}=\Zset^n+\vec{q}
\ee
for some $\vec{q}\in \Rset^n$.  Here we do not pursue the optimal assumptions on  $f$.  
Consider the $(n+1)$-dimensional shear flow
$$
V(X)=(0,0, \ldots, 0, Af(x)),
$$
where $X=(x,x_{n+1})\in \Rset^{n+1}$, $x\in \Rset^n$ and the constant $ A\geq 0$ is the flow intensity.  
Note that $n=2$ is the only case that has a clear physical meaning.   
Then the mean curvature type equation \eqref{ge3} in $(n+1)$ dimensions  is reduced to an $n$-dimensional minimal surface type equation as follows. 
Let $G(X,t)=v(x,t)+P\cdot X$ for $X$ and $P=(p,p_{n+1})\in \Rset^{n+1}$ subject to $p_{n+1}\not=0$. 
Then, $v$ satisfies
\be\label{ge3n}
\begin{cases}
v_t+\left(1-d\,{\rm div}{\frac{p+Dv}{\sqrt{p_{n+1}^2+|p+Dv|^2}}}\right)_+\sqrt{p_{n+1}^2+|p+Dv|^2}+A p_{n+1} f(x)=0,\\[3mm]
v(x,0)=0.
\end{cases}
\ee
Note that if $p_{n+1}=0$, the above equation has the trivial solution $v(x,t)=-|P|t$.  The usual comparison principle implies that, for $t\geq 0$,
\be\label{gradient-bound}
\|Dv(\cdot,t)\|_{L^{\infty}(\Rset^n)}\leq A |p_{n+1}|\|Df\|_{L^{\infty}(\Rset^n)}t.
\ee

\begin{defin}\label{defin:set} For fixed $P\in \Rset^{n+1}$,  denote by $S_H=S_H(P)$ the collection of $A\geq 0$ such that  there exists a constant $\overline H_+(P,d,A)$ satisfying that
\be\label{long-time-limit}
-\lim_{t\to +\infty}{v(x,t)\over t}=\overline H_+(P,d,A) \quad \text{for all  $x\in \Rset^n$}.
\ee
\end{defin}
Note that if $p_{n+1}=0$, then $S_H(P)=[0,\infty)$ and $\overline H_+(P,d,A)=|P|$.  Owing to \eqref{gradient-bound}, the convergence (\ref{long-time-limit}), if holds, must be uniform convergence. 
 Write
$$
F(P)=\max_{X\in \Rset^{n+1}}P\cdot V(X)=\max_{x\in \Rset^n}(p_{n+1}f(x)),
$$
which is  the maximum driven force by the flow along direction $P$. 
Thanks to \eqref{c1},  the maximum of $p_{n+1}f(x)$ is attained on $\Zset^{n}$ if $p_{n+1}>0$ and is attained on $\Zset^n+\vec{q}$ if $p_{n+1}<0$.  

Our main result says that when the flow intensity $A$  surpasses a bifurcation point,  the effective burning velocity $\overline H_+(P,d, A)$ does not exist.  

\begin{theo}\label{theo:main1}   
Assume $n\geq 2$. For fixed $d>0$ and a vector $P\in \Rset^{n+1}$ with $p_{n+1}  \neq 0$,  there exists a finite number $A_0=A_0(P) \geq {1\over 2\max_{\Rset^n}|f|}$ such that 
$$
S_H(P)=[0,A_0].
$$
Moreover, we have the following several properties.
\begin{itemize}
\item[(1)]  As a function of $A$, $\overline H_+(P,d,A) - F(P)A:[0,A_0]\to [0,|P|]$ is  Lipschitz continuous  on $[0,A_0]$ subject to $\overline H_+(P,d,0)=|P|$ and $\overline H_+(P,d,A_0)=F(P)A_0$;

\item[(2)]  There exists $A_1=A_1(P)\in(0, A_0]$  that is continuous with respect to $d$, $P$ and $f$ such that  $\overline H_+(P,d,A)- F(P)A$ is strictly decreasing on $[0,A_1]$ and $\overline H_+(P,d,A)=F(P)A$ for $A\in [A_1, A_0]$.  In addition,  for $A\in [0,A_1]$,  the corresponding cell problem \eqref{cell-cutoff-2} has a $C^{2,\alpha}$ $\Zset^n$-periodic  solution for any $\alpha\in (0,1)$.

\item[(3)] When $n=2$ (the physical case),   $A_0=A_1$. In particular, this implies that the transition value $A_0(P)$ is stable with respect to $d$,  $P$ and $f$. 
\end{itemize}

\end{theo}
\begin{center}\label{curv4}
\includegraphics[scale=0.6]{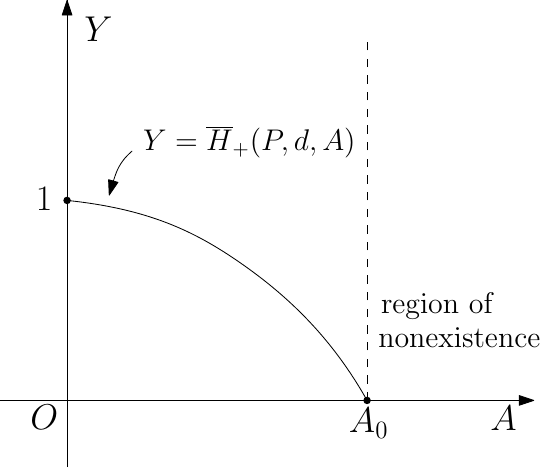} 
\captionof{figure}{Graph of $\overline H_+(P,d,A)$ when $n=2$, $|P|=1$,  $p_3>0$, $\max_{\Rset^2}f=0$ (and hence, $F(P)=0$)}
\end{center}

\begin{rmk}
Some comments are in order.

\begin{itemize}

\item[I.] {\bf Bifurcation of full homogenization.} 
For the shear flow, the existence of $\overline H_+(P)$ along all directions $P$ implies strong homogenization of the curvature G-equation \eqref{ge1} via the standard perturbed test function method.  
See Theorem \ref{theo:hom} in the Appendix.  
Accordingly, $\bar A=\inf_{|P|=1}A_0(P)$ is  the bifurcation point for strong homogenization.  
Precisely speaking,   for $A\in [0, \bar A)$ (homogenization zone), the strong homogenization holds for any uniformly continuous initial data $G(x,0)=g(x)$.  
For $A>\bar A$, the strong homogenization will fail for at least one linear initial data.   
When $n=2$, the homogenization zone is $[0, \bar A]$ due to the continuity of $A_0(P)$.  
For $n\geq 3$, it is not clear what happens when $A=\bar A$.

\item[II.]  {\bf Determination of the bifurcation value $A_0=A_0(P)$ for three dimensional shear flows.}  
When the physical correction $()_+$ is absent (the non-cutoff equation), the corresponding limit (denoted by $\overline H(P,d, A)$)  and homogenization always hold for all $P\in \Rset^{n+1}$ and $A\geq 0$,  which is a special case of the homogenization result in \cite{LS2005}. 
Moreover, $\overline H(P,d, A)- F(P)A$ is strictly decreasing for $A\in [0, \infty)$.  Below is the relation between $\overline H$ and $\overline H_+$:
$$
\overline H_+(P,d,A)=\max\{\overline H(P,d,A),\  AF(P)\}  \quad \text{for $A\in [0,A_0]$}.
$$
See Section \ref{sec:3} for more details.  The value   $A_1=A_1(P)$ in Theorem \ref{theo:main1} is the unique number such that $\overline H(P,d, A_1)- F(P)A_1=0$. Then
$$
\overline H(P,d, A)=\overline H_+(P,d, A) \quad \text{for $A\in [0,A_1]$}.
$$
In two dimensions, we have proved that $A_0=A_1$,  which  provides a way to determine the transition value in the physical situation.  Here is an intuitive picture to look at this physical meaning assuming that $\max_{\Rset^2}f=0$ and $P=(0,0,1)$:  When the flow intensity reaches $A_0$, the effective flame front (or the corresponding traveling wave solution \eqref{travel-wave}) becomes stagnant, and is immediately tore apart and destroyed if the flow intensity is increased a bit further.

\item[III.]  {\bf Whether $A_0=A_1$ when $n\geq 3$.}  
The cutoff equation can be more or less viewed as the non-cutoff equation with isolated singularities.  
Mathematically,    whether $A_0=A_1$ is closely related to the removabilities of those isolated singularities, which requires deep regularity theory from minimal surface type equations.   
Although $n\geq 3$ has no clear physical meaning in our context, it remains an interesting mathematical question whether $A_0=A_1$. Our proof of (2) relies on a Harnack inequality estimate in \cite[Theorem 1']{S1} that is only known to be true in two dimensions.  
Without the requirement of  \eqref{c1}, it is possible that $\overline H(P,d, A)<\overline H_+(P,d, A) $, which also leads to the instability of the transition value $A_0(P)$.   See Lemma \ref{lem:comparison} and Remark \ref{rmk:smooth}.  

\end{itemize}
 
 \end{rmk}

\subsection{Further questions to consider}

\begin{itemize}

\item {\bf Question 1.}  There is a consensus in the combustion literature that the curvature effect slows down flame propagation \cite{Ron}. 
Heuristically, this is because the curvature term smooths out the flame front and reduces the total area of chemical reaction \cite{Se1985}. The first mathematically rigorous result in this direction was obtained in \cite{LXY} for two dimensional shear flows through a delicate analysis, which implies that  for any non-constant  two dimensional shear flow, $d,A>0$, and any unit direction $P\in \Rset^2$, 
$$
\overline H_+(P,d,A)<\overline H_+(P,0,A),
$$
which is consistent with the experimental observation (e.g., \cite{CWL2013}). See Remark  \ref{rmk:1d} for more details.   An interesting question is to prove or disprove this for three dimensional shear flows. When $P=e_3$, the above inequality is easy to verify (see (5) in  Lemma \ref{lem:prop}).   However, it is not clear for other directions $P$.

\medskip

\item {\bf Question 2.}  It is not clear to us whether the bifurcation  is a generic  phenomenon in three dimensions.  
Unlike shear flows,   most realistic three dimensional  incompressible flows possess swirls and chaotic structures that help better mix things.  A famous periodic model is the Arnold-Beltrami-Childress (ABC) flow  \cite{CG}, a steady periodic solution to the Euler equation,
$$
V(X)=(a\sin x_3+c\cos x_2,\  b \sin x_1+a\cos x_3, \ c\sin x_2+b\cos x_1), 
$$
where $a$, $b$ and $c$ are three positive parameters. It remains a very interesting question to see whether $S_H=[0,\infty)$ (i.e., homogenization always holds) for this kind of flows.  

\end{itemize}

%
%
%
%
%
%

\subsection{Other related works.} 
When the curvature effect is ignored ($d=0$), the existence of the effective burning velocity and full homogenization were proved independently in \cite{CNS} and \cite{XY2010} by distinct approaches  for general periodic nearly incompressible flows in any dimension.  
In other practical situations (e.g.,  phase transitions in material science \cite{CB,L1980}),  the curvature effect is also considered and  the  motion law is given by (with no drift term and $()_+$ correction)
$$
v_{\vec{n}}=a(x)-d\kappa
$$
for a continuous positive $\Zset^n$-periodic function $a(x)$.  The above formula is known as 
the Gibbs-Thomson relation. The $()_+$ is not needed in crystal growth since both 
freezing and melting could occur in the situation of ice formation \cite{L1980}.  The corresponding equation is 
$$
u_t + \left(a(x)- d\,  {\rm div}\left({Du\over |Du|}\right)\right)|Du|=0.
$$
 A related equation arising from the crystal growth model of birth-and-spread type was studied in \cite{GMT, GMOT}. 

\smallskip

Let us now focus on the periodic homogenization theory for level-set  mean curvature flows from front propagations.

\medskip

$\bullet$ When $a(x)$ and $Da(x)$ satisfy a coercivity condition, homogenization was proved in \cite{LS2005} for all dimensions thanks to the uniform gradient bound for the correctors (see Lemma \ref{lem:estimate} for a special case).  Employing the same coercivity condition in  \cite{LS2005},  an optimal rate of convergence $O(\epsilon^{1/2})$ can be established in  the special laminar setting \cite{Jang}. 

\medskip

$\bullet$  When $n=2$, homogenization was proved in \cite{CM} for all positive $a(x)$ by a geometric approach.
Moreover,  counterexamples are constructed  there when $n\geq 3$.  

\medskip

$\bullet$  Whether the homogenization still holds without coercivity on the gradient variable  is an  interesting problem in this area, where  approaches different from \cite{CM, LS2005} are needed.   There are various scenarios of non-coercivity. 
For example,   see \cite{CLS} for cases where $a(x)$ changes its signs.  
With the presence of a large drift term,  the curvature G-equation represents one of the most physically relevant non-coercive  examples.  
In  \cite{GLXY},    the cell problems in the homogenization theory (and consequently, the effective quantities)  have been established for an important class of two dimensional incompressible periodic flows  via a new Lagrangian method.

See \cite{ CLS, CN, DKY, KG} and the references therein for other  works in this aspect.  

\subsection*{Organization of the paper}

The paper is organized as follows.
In Section \ref{sec:2}, we give the proof of Theorem \ref{theo:main1}. The proofs in Section \ref{sec:2} rely on results from Section \ref{sec:3}, where we consider both the non-cutoff and cutoff equations and compare the corresponding effective burning velocities. In the Appendix, for the reader's convenience, we present a proof of the homogenization of the curvature G-equation with shear flows under the assumption that \eqref{long-time-limit} holds for all $P\in \Rset^{n+1}$ (equivalently, all approximate cell problems have solutions)  via standard methods.

\section{Proof of  Theorem \ref{theo:main1}} \label{sec:2}
Let $P=(p,p_{n+1})\in \Rset^{n+1}$ and $p\in \Rset^n$. This entire section is dedicated to the proof of  Theorem \ref{theo:main1}. Without loss of generality, we assume 
$$
p_{n+1}=1.
$$ 
Then $|P|=\sqrt{1+|p|^2}$.  Also,  by subtracting some suitable constants, we can always assume that
\[
\max_{\Rset^n}f=0.
\]
 This implies that $F(P)=0$. 

\subsection{Proof of items (1) and (2) of Theorem \ref{theo:main1} }
For  $\lambda>0$,  denote by $v_{\lambda}\in C(\Rset^n)$ the unique periodic viscosity solution to
\be\label{lambda-eq}
\lambda v_{\lambda}+\left(1-d\, {\rm div}{\frac{p+Dv_{\lambda}}{\sqrt{1+|p+Dv_{\lambda}|^2}}}\right)_+\sqrt{1+|p+Dv_{\lambda}|^2}+Af(x)=0.
\ee
 By  the maximum principle and the comparison principle,  
\be\label{eq:bdd-u-lambda}
\begin{cases}
 -|P|\leq  \lambda v_{\lambda}\leq -|P|-A\min_{\Rset^n}f,\\
\lambda \|Dv_{\lambda}\|_{L^{\infty}(\Rset^{n})}\leq A \| Df\|_{L^{\infty}(\Rset^{n})}.
\end{cases}
\ee
It is well known that (see \cite{AB} for instance),   the validity of \eqref{long-time-limit} is equivalent to  
\be\label{ergodic-limit}
-\lim_{\lambda\to 0}\lambda v_{\lambda}(x)=\overline H_+(P,d,A)\quad \text{uniformly in $\Rset^n$}.
\ee
Then $S_H=S_H(P)$ from Definition \ref{defin:set} is also the set of $A\geq 0$ such that there exists a constant (denoted by $\overline H_+(P,d,A)$) so that  the above \eqref{ergodic-limit} holds.  Also, throughout this section,  denote by $\overline H(P,d,A)$  the corresponding quantity without the cutoff ,  where, unlike the cutoff case,  the corresponding limit  \eqref{no-cut-limit} always exists for all $P\in \Rset^{n+1}$ and $A\geq 0$.  See section 3 for details. In particular,  owing to (2) in  Lemma  \ref{lem:formula}, if \eqref{ergodic-limit} holds, then we must have
\be\label{connection}
\overline H_+(P,d,A)=\max\{\overline H(P,d,A),\ AF(P)\}.
\ee
Hereafter,  for simplicity, we write $\overline H_+(d,A)=\overline H_+(P,d,A)$ (similarly, $\overline H(d,A)=\overline H(P,d,A)$)  as there is no confusion.

\begin{lem} $S_H=[0,\infty)$ or $S_H=[0,A_0]$ for some finite number $A_0\geq {1\over 2\max_{\Rset^n}|f|}$ with $\overline H_+(d,A_0)=0$.
\end{lem}

\begin{proof}  Recall that $F(P)=0$.   Owing to  (1) in Lemma \ref{lem:prop},   $\overline H(P,d,A)\geq {1\over 2}$ if $A\in \left[0, \ {1\over 2\max_{\Rset^n}|f|}\right]$. Hence  (1) in Lemma \ref{lem:formula} implies that the interval $\left[0, \ {1\over 2\max_{\Rset^n}|f|}\right]\subset S_H$.   For $\lambda>0$ and $i=1,2$,   suppose that $v_{\lambda, i}(x)$  are solutions to \eqref{lambda-eq} corresponding to $A=A_i$ respectively.  
By the comparison principle, 
\be\label{A-dependence}
\lambda\|v_{\lambda,1}-v_{\lambda,2}\|_{L^\infty(\Rset^n)}\leq |A_1-A_2|\max_{\Rset^n}|f|.
\ee
Hence $S_H$ is a closed set.   Due to Corollary \ref{cor:order},  $S_H=[0,\infty)$ or $S_H=[0,A_0]$ for some $A_0\geq {1\over 2\max_{\Rset^n}|f|}$.  

 To finish the proof,  we just need to show that if $S_H=[0,A_0]$, then $\overline H_+(d,A_0)=0$. 
 In fact, if $\overline H_+(d,A_0)>0$,  then  by \eqref{connection},  $\overline H_+(d,A_0)=\overline H(d,A_0)>0$. Thanks to the continuous dependence of $\overline H(d,A)$ on $A$, there exists $\tau>0$ such that $\overline H(d,A)>0$ for $A\in [A_0-\tau, A_0+\tau]$.  Due to  (1) in  Lemma \ref{lem:formula},  $ [A_0-\tau, A_0+\tau]\subset S_H$, which is a contradiction. 
 \end{proof}

It remains to show that $S_H\not=[0,\infty)$ to prove (1). 
This is where we need crucially assumption \eqref{c1}.   
Choose 
$$
r_0={\min\left\{{1\over 1+|p|},\ d\right\}\over 8}.
$$
Denote  
$$
K_1=\min_{\Rset^n\backslash (B_{r_0}(0)+\Zset^n)}|f|>0
$$
and
$$
K_2=\max_{|x|\leq r_0}\left[\left(1-d\, {\rm div}{\frac{p+D\tilde P(x)}{\sqrt{1+|p+D\tilde P(x)|^2}}}\right)\sqrt{1+|p+D\tilde P(x)|^2}\right]>0,
$$
where
$$
\tilde P(x)=2-{2|x|^2 \over r_0^2}.
$$
Write
$$
\tilde A={1+K_2+|P|\over K_1}. 
$$

\begin{lem} 
$$
[\tilde A, \infty)\cap S_H=\emptyset.
$$
\end{lem}

\begin{proof}
We proceed the proof by contradiction.  
Fix $A\geq \tilde A$ and assume that $A\in S_H$.
Then
$$
-\lim_{\lambda\to 0}\lambda v_{\lambda}(x)=c\quad \text{uniformly in $\Rset^n$}
$$
for some constant $c\in \Rset$.  Clearly,
$$
c\geq 0.
$$

Choose $\lambda>0$ small enough such that 
$$
\max_{\Rset^n}\lambda v_{\lambda}\leq {1\over 2}.
$$
Then $v_\lambda$ is a viscosity supersolution to
\be \label{eq:super}
\left(1-d\,{\rm div}{\frac{p+Dv_{\lambda}}{\sqrt{1+|p+Dv_{\lambda}|^2}}}\right)_+\sqrt{1+|p+Dv_{\lambda}|^2}+Af(x)\geq -{1\over 2}.
\ee

In the following, we only use the fact that $v_\lambda$ is a supersolution to \eqref{eq:super}.
In particular, there is no need to use \eqref{lambda-eq} in our arguments below.
Since $AK_1\geq 1+ |P|$,  owing to  the maximum principle (or the viscosity supersolution test), when  $\lambda>0$ is small enough,  the minimum of  $v_{\lambda}$ on $\Rset^n\backslash \left(B_{r_0}(0)+\Zset^n\right)$ can only be attained  on $\partial B_{r_0}(0)+\Zset^n$.  Also, since we are only concerned with  \eqref{eq:super}, by subtracting an appropriate constant from $v_\lambda$, we may without loss of generality  assume that 
\be\label{above-zero}
\min_{\Rset^n\backslash (B_{r_0}(0)+\Zset^n)}v_{\lambda}= \min_{\partial B_{r_0}(0)+\Zset^n}v_{\lambda}=0.
\ee
We claim that
\be\label{boundary-control}
v_{\lambda}(x)\geq 2\quad \text{for $x\in \partial B_{2r_0}(0)+\Zset^n$}.
\ee
We argue by contradiction.   Assume that there exists $x_0\in \partial B_{2r_0}(0)$ such that 
$$
v_\lambda(x_0)<2.
$$
Denote
$$
P(x)=2-{2|x-x_0|^2\over r_0^2}.
$$
Then 
$$
v_{\lambda}(x_0)-P(x_0)<0 \quad \mathrm{and} \quad v_{\lambda}(x)-P(x)\geq 0 \quad \text{for $|x-x_0|= r_0$}.
$$
Thus,  $v_{\lambda}(x)-P(x)$  attains a local minimum at some point $\bar x\in B_{r_0}(x_0)$. 
See Figure 3 below.  
By the viscosity supersolution test, 
\begin{align*}
K_2-AK_1&\geq K_2+Af(\bar x)\\
&\geq \left(1-d\, {\rm div}{\frac{p+DP(\bar x)}{\sqrt{1+|p+DP(\bar x)|^2}}}\right)\sqrt{1+|p+D P(\bar x)|^2}+Af(\bar x)\geq -{1\over 2}, 
\end{align*}
which contradicts to $AK_1\geq K_2+1$.   Hence our claim (\ref{boundary-control}) holds. 

\begin{center}\label{curv1}
\includegraphics[scale=0.6]{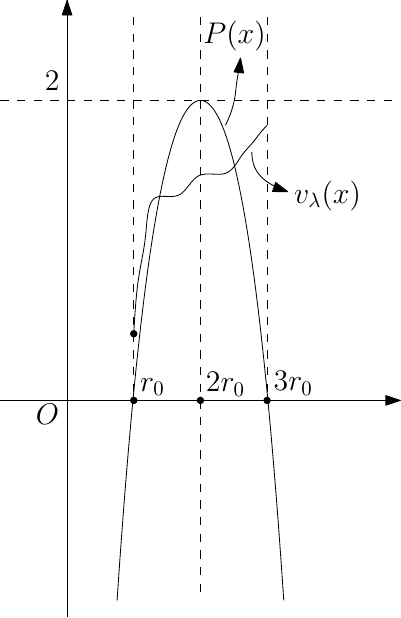} 
\captionof{figure}{The graph of $P(x)$ and $v_{\lambda}(x)$}
\end{center}

We note further that $u_\lambda(x)=p\cdot x+ v_\lambda(x)$ is a viscosity supersolution of 
\be\label{eq:supern}
\left(1-d\,{\rm div}{\frac{Du_{\lambda}}{\sqrt{1+|Du_{\lambda}|^2}}}\right)\sqrt{1+|Du_{\lambda}|^2}\geq 1 \quad \text{for $r_0<|x|<3r_0$}.
\ee
We now consider a radially symmetric lower envelope of $u_\lambda$.
Define
$$
w(x)=\min_{|y|=|x|}u_{\lambda}(y).
$$
 Thanks to (\ref{eq:bdd-u-lambda}), $w$ is Lipschitz continuous. Since the operator involved in \eqref{eq:supern} is rotation invariant,   $w$ is also a viscosity supersolution of \eqref{eq:supern}.  Thus,  $w(x)=g(|x|)$ for a  Lipschitz  continuous function $g:\left[r_0, {1\over 2}\right]\to \Rset$ that is a viscosity supersolution of 
$$
\sqrt{1+(g')^2}-d\left({g''\over 1+(g')^2}+{(n-1)g'\over r}\right)\geq 1 \quad \text{for $r\in \left(r_0,\  3r_0\right)$}.
$$
Due to \eqref{above-zero} and  \eqref{boundary-control},  $g(r_0)\leq  r_0|p|<{1\over 4} $ and $g({2r_0})\geq 2-2{r_0}|p|>{3\over 2}$.  In particular, $g(2r_0)-g(r_0)>1$. 

The main point of choosing $A$ large ($A \geq \tilde A$ in the above) is to get that $g$ grows rather quickly in the small interval $[r_0,2r_0]$.
Due to Lemma \ref{lem:barrier} below and $2r_0-r_0=r_0<{1\over 4}$,  there exist $L\geq 1$ and $\bar r\in  (r_0, 2r_0)$ such that
$$
\sqrt{1+L^2}-d\left(-2L+{(n-1)L\over \bar r}\right)\geq 1.
$$
This leads to $1\geq d\left({n-1\over \bar r}-2\right)\geq d\left({1\over \bar r}-2\right)$, which contradicts  the choice of $r_0$. 
The main term helping us to reach the contradiction is the curvature term $-d \frac{(n-1)g'}{r}$.
\end{proof}

\begin{lem}\label{lem:barrier} 
Assume that $\delta\in (0, {1\over 4}]$ and  $h:[0,\delta]\to \Rset$ is a  Lipschitz continuous function satisfying that  $h(0)=0$ and $h(\delta)=1$.  
Then there exist  $\bar r \in (0,\delta)$ and a smooth function $l:[0,\delta]\to \Rset$ such that  $h-l$ attains a local minimum at $\bar r$ and
$$
l'(\bar r)\geq {1\over 4\delta} \quad \mathrm{and} \quad {l''(\bar r)\over 1+(l'(\bar r))^2}\geq -2l'(\bar r).
$$
\end{lem}

\begin{proof}
Choose an arbitrary $\alpha$ such that
$$
0<\alpha<{1\over 2}\min\left\{r\in [0,\delta]: \ h(r)={1\over 2}\right\}.
$$
Consider the $(r,s)$ coordinate plane and a family of parabolas $\Gamma_{\theta}$
$$
 r=r(s)=\delta+{\alpha}+\theta s(s-1)
$$
for $s\in [0,1]$ and $\theta\in [0,  {4\delta}]$.   
The graph of $h$ can be expressed as   $\{(r,s)\,:\, s=h(r)\}$.   
Increasing $\theta$,  apparently $\Gamma_{\theta}$ will touch the graph of $h$ from the right hand side at  $(\bar r, h(\bar r))$ for some $h(\bar r)\in \left({1 \over 2}, 1\right)$.
Then $\tau=\bar r -({\delta+\alpha- {\theta\over 4}})>0$. See Figure 4 below.  
Let $l$ be a smooth function on $[0,\delta]$ such that
$$
l(r)=\sqrt{{r-\delta-\alpha\over \theta}+{1\over 4}}+{1\over 2}
$$
for  $r\in (\bar r-{\tau\over 2}, \delta]$. Then $h-l$ attains a local minimum at $\bar r$ and 
$$
l'(\bar r)={1\over 2\theta}\left({\bar r-\delta-\alpha\over \theta}+{1\over 4}\right)^{-{1\over 2}}\geq {1\over \theta}\geq {1\over 4\delta}
$$
and
$$
l''(\bar r)= -{1\over 4\theta^2}\left({\bar r-\delta-\alpha\over \theta}+{1\over 4}\right)^{-{3\over 2}}\geq -2(l'(\bar r))^3.
$$
\end{proof}

\begin{center}\label{curv2}
\includegraphics[scale=0.6]{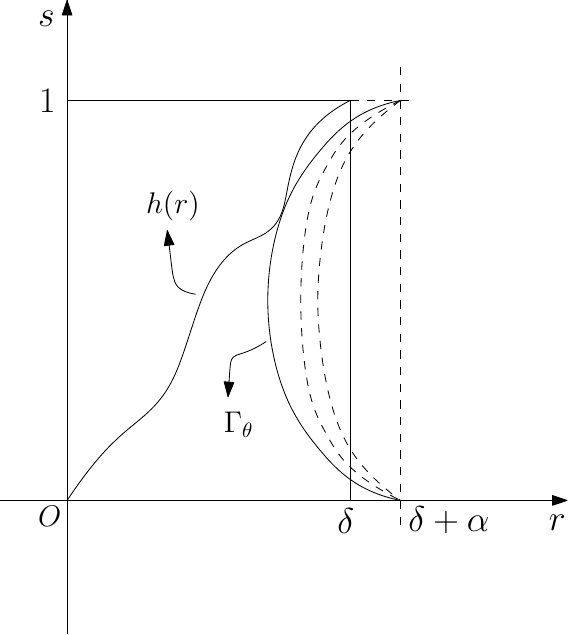} 
\captionof{figure}{Touching of $\Gamma_\theta$ and the graph of $h=h(r)$}
\end{center}

By similar arguments, we can actually derive the following fact related to isolated singularities.  
See \cite{S2} for more general results concerning minimal surface type equations when the set of  singularities has zero $n-1$ dimensional Hausdorff measure, where the proof is much more involved.

\begin{cor}\label{cor:bound}  Suppose that $\hat u\in C(B_1(0)\backslash\{0\})$ is a viscosity supersolution of 
$$
\left(1-d\,{\rm div}{\frac{D\hat u}{\sqrt{1+|D\hat u|^2}}}\right)\sqrt{1+|D\hat u|^2}\geq C \quad \text{in $B_1(0)\backslash\{0\}$}.
$$
for some constant $C$.  Then $\hat u$ has a lower bound, i.e.,
$$
\inf_{B_1(0)\backslash\{0\}}\hat u>-\infty.
$$
\end{cor}

Finally, we explain how to determine the value $A_1$ for fixed $P\in \Rset^{n+1}$ with $p_{n+1}\not= 0$.  To demonstrate the general rule, we keep $F(P)$ in the discussion here instead of replacing it by 0 as the assumption at the beginning of this section.  Recall that  $\overline H(P,d,A)$ is the effective quantity from the equation without cutoff (see Section 3).  
Note that  $\overline H(P,d,A)-AF(P)$ must be negative  for some $A> 0$. Otherwise,  $S_H(P)=[0,\infty)$ by (1) in Lemma \ref{lem:formula}.  So,  owing to (2) and (3) in Lemma \ref{lem:prop},  there exists a unique number $A_1=A_1(P)\in (0, A_0]$ such that 
$$
\overline H(P, d, A_1)=A_1F(P).
$$
Then
$$
\overline H(P, d, A)
\begin{cases}
>AF(P) \quad \text{for $A\in [0, A_1)$}\\[3mm]
<AF(P) \quad \text{for $A<A_1$}.
\end{cases}
$$
 Clearly, $A_1(P)$ is continuous with respect to $d$, $P$  and $f$.  Therefore, \eqref{connection} implies that
 $$
\overline H_+(P, d, A)=
\begin{cases}
\overline H(P, d, A)>AF(P) \quad \text{for $A\in [0, A_1)$}\\[3mm]
AF(P) \quad \text{for $A\in [A_1, A_0]$}.
\end{cases}
$$
The Lipschitz continuity with respect to $A$ is from \eqref{A-dependence}. 
Apparently,  for $A\in [0, A_1]$, the cell problems of equation without the physical cutoff \eqref{cell} is the same as the cell problem with the cutoff \eqref{cell-cutoff-2}.  
Hence \eqref{cell-cutoff-2} has a unique  $C^{2,\alpha}$ periodic solution $v=v(x,P)$ (up to a constant) for any $\alpha\in (0,1)$, equivalently a traveling wave solution for \eqref{ge1}:
\be\label{travel-wave}
G(X,t)=P\cdot X+v(x,P)-\overline H_{+}(P,d,A)t.
\ee

When $n\geq 3$, it is not clear to us whether $A_1=A_0$. If not, then the bifurcation value $A_0=A_0(P)$ might not be stable with respect to $f$. See  Remark \ref{rmk:smooth}.

\subsection{Proof of item (3) of Theorem \ref{theo:main1} }

Next,  we prove $A_0=A_1$ when $n=2$ (three dimensional shear flows),  which implies the stability of the transition value.  Let $\tilde v$ be a  $C^{2,\alpha}$ periodic solution to  the cell problem of the non-cutoff equation (see Section \ref{sec:3} below for the existence of $\tilde v$)
\be\label{cell-non-cutoff}
\left(1-d\, {\rm div}{\frac{p+D\tilde v}{\sqrt{1+|p+D\tilde v|^2}}}\right)\sqrt{1+|p+D\tilde v|^2}+Af(x)=\overline H(d,A) \quad \text{in $\Rset^n$}.
\ee
Recall that,  for $\lambda>0$, $v_\lambda\in C(\Rset^2)$ is the unique  periodic viscosity solution of \eqref{lambda-eq}. 
 Note that if $\overline H(d,A)\geq 0$, then $\tilde v$ is also a solution to the cell problem of the cutoff equation
\be\label{cell-cutoff1}
\left(1-d\, {\rm div}{\frac{p+D\tilde v}{\sqrt{1+|p+D\tilde v|^2}}}\right)_+\sqrt{1+|p+D\tilde v|^2}+Af(x)=\overline H_+(d,A)=\overline H(d,A) \quad \text{in $\Rset^n$},
\ee
which implies that for the solution $v_\lambda$ to (\ref{lambda-eq}), 
$$
\lim_{\lambda\to 0}-\lambda v_{\lambda}(x)=\overline H(d,A)\quad \text{uniformly in $\Rset^2$}.
$$

By (3) in Lemma \ref{lem:prop},   $\overline H(d,A)$ is strictly decreasing with respect to $A$ for all $A>0$.  Also, note that $\overline H_+(P,d,A)=0$ for $A\in [A_1,A_0]$. Therefore, to establish $A_1=A_0$, it suffices to prove the following lemma. 
\begin{lem}\label{lem:sign} If
\be\label{limit}
\lim_{\lambda\to 0}\lambda v_{\lambda}(x)=0 \quad \text{uniformly in $\Rset^2$},
\ee
then $\overline H(d,A)= 0$.
\end{lem}

\begin{proof}
Without loss of generality,  let $A=1$.   
Suppose that  \eqref{limit} holds. 
Then   $\overline H_+(d,1)=0$ and   $\overline H(d,1)\leq 0$. 
So it suffices to show that  $\overline H(d,1)\geq 0$. 
For $\delta \ge 0$,  write
$$
D_\delta=\left\{x\in \Rset^2\,:\, f(x)< -\delta\right\}.
$$
Note that for any fixed $\delta\in (0, {1\over 2}]$, when $\lambda$ is small enough,   $v_\lambda$ is a viscosity solution to 
$$
\lambda v_{\lambda}+\left(1-d\, {\rm div}{\frac{p+Dv_{\lambda}}{\sqrt{1+|p+Dv_{\lambda}|^2}}}\right)\sqrt{1+|p+Dv_{\lambda}|^2}+f(x)=0 \quad  \text{ in $D_\delta$}.
$$
Recall that we have, in light of \eqref{eq:bdd-u-lambda},
$$
\lambda \|Dv_{\lambda}\|_{L^{\infty}(\Rset^2)}\leq \|Df\|_{L^{\infty}(\Rset^2)}.
$$
Accordingly,  by the regularity theory of quasilinear elliptic equations in  \cite{GT},
$$
v_{\lambda}\in C^{2,\alpha}(D_\delta).
$$
for any $\alpha\in (0,1)$.  Choose $x_{\lambda}\in \Rset^2$ such that  $v_\lambda(x_{\lambda})=\max_{\Rset^2}v_{\lambda}$. Then by the definition of viscosity solutions, 
$$
\lambda v_{\lambda}(x_{\lambda})+|P|+f(x_\lambda)\leq 0.
$$
Since $|P|\geq 1$,    when $\lambda$ is small enough, we have that $x_{\lambda}\in D_{1\over 2}$, i.e., :
$$
\{x\in \Rset^2\,:\, v_\lambda(x)=\max_{\Rset^2}v_{\lambda}\}\subset D_{1\over 2}. 
$$
Let $u_\lambda(x)=v_{\lambda}(x)+p\cdot x$.  Since $v_\lambda\in C^{2,\alpha}( D_{\delta})$,  we have $Du_\lambda(x_\lambda)=p$ and 
\be\label{mini-form}
-d\, {\rm div}{\frac{Du_{\lambda}}{\sqrt{1+|Du_{\lambda}|^2}}}={{-f(x)+\lambda p\cdot x-\lambda u_{\lambda}\over \sqrt{1+|Du_{\lambda}|^2}}-1} \quad \text{in $D_\delta$}.
\ee
Using notations from  \cite{S1},  we write here for $1\leq i, j\leq 2$
$$
a_{ij}^{*}(x,z,\mu)=\delta_{ij}-\mu_i\mu_j \quad \text{and} \quad b^*(x,z,\mu)=(f(x)-\lambda p\cdot x+\lambda z)\mu_3+1,
$$
where $\mu=(\mu_1, \mu_2, \mu_3)={(-p,\ 1)\over \sqrt{1+|p|^2}}$.  
Since $\lambda z$ is uniformly bounded in our situation,  it is easy to check that $a_{ij}^{*}$ and $b^*(x,z,\mu)$ satisfy the structural requirements (3.2) and (3.3) in \cite{S1}.  
Then  we  apply the Harnack inequality \cite[Theorem 1']{S1}  to equation \eqref{mini-form} to obtain that
$$
\|Du_{\lambda}\|_{L^{\infty}(D_\delta)}\leq C_\delta.
$$
Hereafter $C_\delta>0$ is a constant depending only on $\delta$, $f$ and $|p|$.  This is the only place the dimension 2 is used.   Set
$$
\tilde u_{\lambda}=u_{\lambda}-u_{\lambda}\left({1\over 2}, {1\over 2}\right).
$$
The regularity theory of quasilinear elliptic equations in  \cite{GT} implies that $\|\tilde u_\lambda\|_{C^{2,\alpha}(D_\delta)}\leq C_{\delta}$.  By passing to a subsequence if necessary,  we have that 
$$
\lim_{\lambda\to 0}\tilde u_{\lambda}=u(x) \quad \text{locally uniformly on $D_0$.}
$$
Here, $u\in C^{2,\alpha}(D_0)$ is a periodic solution to 
$$
-d\, {\rm div}{\frac{Du}{\sqrt{1+|Du|^2}}}={-f(x)\over \sqrt{1+|Du|^2}} -1\quad \text{in $D_0$},
$$
and $u\left({1\over 2}, {1\over 2}\right)=0$. 
In particular,  $\Zset^2$ is the set of isolated singularities of $u$.  
Due to Corollary \ref{cor:bound},  
$$
\liminf_{x\to 0}u(x)>-\infty.
$$
In other words,  $u$ is bounded.  Now define
$$
u(0)=\liminf_{x\to 0} u(x).
$$
Then $u$ is lower semi-continuous at $0$.   Choose $\bar x\in \Rset^2$ satisfying  that for $\tilde u=p\cdot x+\tilde v$ with $\tilde v$ from \eqref{cell-non-cutoff}, 
$$
u(\bar x)-\tilde u (\bar x)=\min_{\Rset^2}(u-\tilde u).
$$
Thanks to Lemma \ref{lem:supersolution} below,
$$
-d\, {\rm div}{\frac{D\tilde u}{\sqrt{1+|D\tilde u|^2}}}(\bar x)\geq {-f(\bar x)\over \sqrt{1+|D\tilde u(\bar x)|^2}} -1.
$$
Equivalently,
$$
\left(1-d\, {\rm div}{\frac{D\tilde u}{\sqrt{1+|D\tilde u|^2}}}(\bar x)\right)\sqrt{1+|D\tilde u|^2(\bar x)}+f(\bar x)\geq 0.
$$
Then  \eqref{cell-non-cutoff} leads to $\overline H(d,1)\geq 0$.  \end{proof}

Write
$$
H(x)={-f(x)\over \sqrt{1+|Du|^2}} -1.
$$
Then $H(x)$ is periodic and continuous with $H(0)=-1$.  
The following lemma says that the isolated singularities  of solutions to $-d\, {\rm div}{\frac{Du}{\sqrt{1+|Du|^2}}}=H(x)$ at $0$ are removable in the sense of viscosity solutions.  
If $H(x)$ is Lipschitz continuous, it is a classical result of  \cite{S2}  that a set  of  singularities with zero ($n-1$)-dimensional Hausdorff measure is removable.  
In our situation, it is not clear whether $H(x)$ is Lipschitz continuous near $\Zset^2$.    A weaker version in the sense of viscosity solutions is enough for our purpose. 

\begin{lem}\label{lem:supersolution}  
Suppose that $\phi$ is  a $C^2$ function such that  $u-\phi$ attains a   minimum at $\tilde x\in \Rset^2$. 
Then
$$
-d\, {\rm div}{\frac{D\phi}{\sqrt{1+|D\phi|^2}}}(\tilde x)\geq {-f(\tilde x)\over \sqrt{1+|D\phi(\tilde x)|^2}} -1.
$$
Here $u$ is the function in the proof of the previous Lemma \ref{lem:sign}. 
\end{lem}

\begin{proof}
 The conclusion is obvious when $\tilde x\notin \Zset^2$.  
Let us now assume that $\tilde x=0$. 
Assume by contradiction that
$$
-d\, {\rm div}{\frac{D\phi}{\sqrt{1+|D\phi|^2}}}(0)<-1=H(0).
$$
Then due to the continuity of $H=H(x)$,   there exists $\tau\in (0, {1\over 4})$ such that 
$$
-d\, {\rm div}{\frac{D\phi}{\sqrt{1+|D\phi|^2}}}(x)<H(x)-\tau \quad \text{for $x\in B_{\tau}(0)$},
$$
Without loss of generality, we may assume that
$$
0=u(0)-\phi(0)<u(x)-\phi(x) \quad \text{for $x\in B_{1\over 4}(0)\backslash\{0\}$}.
$$
Choose $\delta>0$ such that 
$$
U_\delta=\left\{x\in B_{1\over 4}(0)\,:\, u(x)<\phi(x)+\delta\right\}\subset B_\tau(0).
$$
Since $u$ is continuous away from 0 and $u(0)=\liminf_{x\to 0}u(x)$,  $U_\delta\backslash\{0\}$ is a non-empty open set.  Note that $U_{\delta}$ might not be open at the point $0$.   Choose $\theta>0$ small enough such that
$$
U_{\delta,\theta}=U_{\delta}\backslash \overline{B_{\theta}(0)}\not=\emptyset. 
$$
Clearly, $U_{\delta,\theta}$ is an open set.   Note that
$$
-d\, {\rm div}{\frac{D\phi}{\sqrt{1+|D\phi|^2}}}+d\, {\rm div}{\frac{Du}{\sqrt{1+|Du|^2}}}<-\tau \quad \text{in $U_{\delta,\theta}$}.
$$
Choose $\eta\in C^{\infty}(\Rset^2,[0,1])$ such that $\eta=0$ in $B_\theta(0)$,  $\eta=1$ on $\Rset^2 \setminus B_{2\theta}(0)$ and $|D\eta|\leq {C\over \theta}$ in $\Rset^2$.   Then
$$
(\phi+\delta-u)\eta=0 \quad  \text{$\partial U_{\delta,\theta}$},
$$
and
$$
\int_{U_{\delta,\theta}}\left(-d\, {\rm div}{\frac{D\phi}{\sqrt{1+|D\phi|^2}}}+d\, {\rm div}{\frac{Du}{\sqrt{1+|Du|^2}}}\right)(\phi+\delta-u)\eta\,dx\leq -\tau \int_{U_{\delta,\theta}}(\phi+\delta-u)\eta\,dx.
$$
Therefore, 
$$
\int_{U_{\delta,\theta}} (F(D\phi)-F(Du))\cdot (D\phi-Du)\eta\,dx\ +\ \mathrm{II}\ \leq -\tau \int_{U_{\delta,\theta}}(\phi+\delta-u)\eta\,dx.
$$
Here, $F(D\psi)=d{\frac{D\psi}{\sqrt{1+|D\psi|^2}}}$ and
$$
\mathrm{II}=d\int_{U_{\delta,\theta}} (\phi+\delta-u)(F(D\phi)-F(Du))\cdot D\eta\,dx.
$$
Due to the boundedness of $u$ and $|F|\leq 1$, 
$$
|\mathrm{II}|\leq C\int_{B_{2\theta}(0)}|D\eta|\,dx\leq  {C\theta}.
$$
Therefore, sending $\theta\to 0$ leads to 
$$
0<\int_{U_{\delta}} (F(D\phi)-F(Du))\cdot (D\phi-Du)\,dx\leq -\tau\int_{U_{\delta}}(\phi+\delta-u)\,dx<0, 
$$
which is absurd. 
The proof is complete.
\end{proof}

\section{Comparison with the non-cutoff equation} \label{sec:3}
In this section, we assume  that $f:\Rset^n\to \Rset$ is $\Zset^n$-periodic and Lipschitz continuous.  
For $\lambda>0$ and $P=(p,p_{n+1})\in \Rset^{n+1}$ subject to $p_{n+1}\not=0$,   consider the  non-cutoff equation
\be\label{ge5}
\lambda \tilde v_{\lambda} + \left(1-d\,{\rm div}{\frac{p+D\tilde v_{\lambda}}{\sqrt{p_{n+1}^2+|p+D\tilde v_{\lambda}|^2}}}\right)\sqrt{p_{n+1}^2+|p+D\tilde v_{\lambda}|^2}+p_{n+1}f(x)=0
\ee
and the cutoff equation
\be\label{ge6}
\lambda  v_{\lambda} + \left(1-d\,{\rm div}{\frac{p+D v_{\lambda}}{\sqrt{p_{n+1}^2+|p+D v_{\lambda}|^2}}}\right)_+\sqrt{p_{n+1}^2+|p+Dv_{\lambda}|^2}+p_{n+1}f(x)=0.
\ee
Here, we first omit the flow intensity constant $A$ and will add it when its role is discussed. 

The cutoff equation can be viewed as the non-cutoff equation with singularities lying on the set $\{p_{n+1}f(x)+\lambda v_{\lambda}=0\}$. Note that if $p_{n+1}=0$, then both the above two equations have trivial solutions $\tilde v_{\lambda}=v_{\lambda}\equiv -{|P|\over \lambda}$. Recall that
 $$
 F(P)=\max_{X\in \Rset^{n+1}}P\cdot V(X)=\max_{x\in \Rset^n}(p_{n+1}f(x)).
 $$

The following  Lipchitz bound is  based on the classical Bernstein method (see, e.g., \cite{LS2005}).  
For the reader's convenience,  we present its proof here. 

\begin{lem}\label{lem:estimate}  Given $D_0>0$, suppose that $d\in (0, D_0]$ and $\lambda\in (0,1]$. Then for $p_{n+1}\not=0$, \eqref{ge5} has a $C^{2,\alpha}$ periodic solution for any $\alpha\in (0,1)$ and
$$
\|D\tilde v_{\lambda}\|_{L^\infty(\Rset^n)}\leq C(1+\|f\|_{W^{1,\infty}(\Rset^n)}).
$$
Here $C$ is a constant depending only on $n, D_0, |P|$.  
\end{lem}

\begin{proof} By abuse of notations, we write $p_{n+1} f(x)$ as $f(x)$ for clarity of presentation. For convenience,  assume that $|P|=1$.   By standard regularity theory of quasilinear elliptic equation \cite{GT},  it suffices to establish the above a priori estimate by assuming that $\tilde v_{\lambda}\in C^{\infty}$.   Write $\phi(x)=p\cdot x+\tilde v_{\lambda}$.  Write $0\not=\tau=p_{n+1}\in [-1,1]$.   Due to the maximum principle, 
$$
-1-\max_{\Rset^n}f\leq \min_{x\in \Rset^n}\lambda \tilde v_{\lambda}(x)\leq \max_{x\in \Rset^n}\lambda \tilde v_{\lambda}(x)\leq -1-\min_{\Rset^n}f.
$$

Let $w=|D\phi|^2$ .   Then
\begin{equation}\label{eq1}
\lambda \phi(x)+\sqrt{\tau^2+w}-d{\Delta \phi}+d{D\phi\cdot Dw\over 2(\tau^2+w)}=-f+\lambda p\cdot x. 
\end{equation}
Assume that for $x_0\in \Rset^n$, 
$$
w(x_0)=\max_{\Rset^{n}}w.
$$
Then $Dw(x_0)=0$, $D^2w(x_0)\leq 0$, and
$$
\Delta w=2(D\phi\cdot D(\Delta \phi)+|D^2\phi|^2).
$$
Taking gradient of the above equation \eqref{eq1},  multiplying by $D\phi$ and evaluating at $x_0$  leads to 
$$
\lambda w(x_0)-d{D\phi(x_0)\cdot D(\Delta \phi)(x_0)}+d{D\phi\cdot D^2w\cdot D\phi\over 2(\tau+w)}(x_0)=-D\phi(x_0)\cdot (Df(x_0)-\lambda p).
$$
Accordingly,
$$
d{|D^2\phi|^2}(x_0)-d{\Delta w(x_0)\over 2}+d{D\phi\cdot D^2w\cdot D\phi\over 2(\tau+w)}(x_0)\leq -D\phi(x_0)\cdot (Df(x_0)-\lambda p).
$$
Since $D^2w(x_0)\leq 0$, 
$$
d{|D^2\phi(x_0)|^2}\leq -D\phi(x_0)\cdot (Df(x_0)-\lambda p)\leq (M+1)\sqrt{w(x_0)}.
$$
Here  $M=\|f\|_{W^{1,\infty}(\Rset^n)}$.  This implies that 
$$
|D^2\phi(x_0)|\leq {\sqrt{M+1}w(x_0)^{1\over 4}\over \sqrt{d}}.
$$
Also,  \eqref{eq1} says that, at $x=x_0$, 
$$
\sqrt{\tau^2+w(x_0)}-d{\Delta \phi(x_0)}=-\lambda\tilde v_{\lambda}(x_0)-f(x_0)\leq 1+2\max_{\Rset^2}|f|. 
$$
So
$$
\sqrt{\tau+w(x_0)}\leq \sqrt{(M+1)d n}\,w(x_0)^{1\over 4}+1+2M\leq \sqrt{(M+1) D_0 n}\,w(x_0)^{1\over 4}+1+2M.
$$
This implies that 
$$
\max_{\Rset^n}|D\tilde v_{\lambda}|\leq \sqrt{w(x_0)}+1\leq C(M+1)
$$
for a constant $C>0$ depending only on $n, D_0$. 
\end{proof}

In particular, we have that
\be\label{no-cut-limit}
\lim_{\lambda\to 0}(-\lambda \tilde v_{\lambda}(x))=\overline H(P,d) \quad \text{uniformly in $\Rset^n$}
\ee
for a unique constant $\overline H(P,d)$.  Also, the following cell problem has a $C^{2,\alpha}$ periodic solution $\tilde v$ for any $\alpha\in (0,1)$
\be\label{cell}
\left(1-d\,{\rm div}{\frac{p+D\tilde v}{\sqrt{p_{n+1}^2+|p+D\tilde v|^2}}}\right)\sqrt{p_{n+1}^2+|p+D\tilde v|^2}+p_{n+1}f(x)=\overline H(P,d) \quad \text{on $\Rset^n$}. 
\ee
Note that $|D\tilde v|$ satisfies the same bound as in Lemma  \ref{lem:estimate}.  Hence it is easy to see that $\overline H$ is continuous with respect to $P$, $d$ and the function $f$. 

Below we will discuss when the limit $\lim_{\lambda\to 0}\lambda v_{\lambda}$ is a constant (denoted by $\overline H_+(P,d)$ as in previous sections), where $v_\lambda$ is the solution to equation \eqref{ge6}.  
The following lemma provides the relation between $\overline H$ and $\overline H_+$.

\begin{lem}\label{lem:formula} For  $P\in \Rset^{n+1}$,  the following two hold. 
\begin{itemize}
\item[(1)]  If $\overline H(P,d)\geq  F(P)$, then
$$
-\lim_{\lambda\to 0}\lambda v_{\lambda}(x)=\overline H(P,d) \quad \text{for all $x\in \Rset^n$}.
$$
\item[(2)]  If
$$
-\lim_{\lambda\to 0}\lambda  v_{\lambda}(x)=\overline H_+(P,d) \quad \text{for all $x\in \Rset^n$},
$$
for a constant $\overline H_+(P,d)$,  then 
\be\label{connection1}
\overline H_+(P,d)=\max \{\overline H(P,d),  \  F(P) \}.
\ee
This implies that   $\overline H_+(P,d)$ is continuous with respect to $P$, $d$ and the function $f$.  Moreover, as a function of $P$, it is positively homogenous of degree 1. 
\end{itemize}
\end{lem}

\begin{proof} (1) follows immediately from the fact that, given $\overline H(P,d)\geq  F(P)$,  $v=\tilde v$ from \eqref{cell} is also a solution to  the cell problem  of the cutoff equation
\be\label{cell-cutoff-2}
\left(1-d\,{\rm div}{\frac{p+Dv}{\sqrt{p_{n+1}^2+|p+Dv|^2}}}\right)_+\sqrt{p_{n+1}^2+|p+D v|^2}+p_{n+1}f(x)=\overline H_+(P,d) \quad \text{on $\Rset^n$}.
\ee

(2) Clearly,  $\overline H_+(P,d)\geq F(P)$. Then the conclusion can be deduced from Lemma \ref{lem:lowerbound} below. 
\end{proof}

\begin{lem}\label{lem:lowerbound} 
For $\lambda>0$,  let $v_{\lambda}$ be the solution to \eqref{ge6}.  
If 
$$
\liminf_{\lambda\to 0}(-\lambda v_{\lambda}(x))\geq F(P) \quad \text{uniformly on $\Rset^n$},
$$
then 
$$
\lim_{\lambda\to 0}(-\lambda v_{\lambda}(x))=\max \{\overline H(P,d),  \  F(P) \} \quad \text{uniformly on $\Rset^n$}.
$$
\end{lem}

\begin{proof}Without loss of generality,  we assume that $\max_{\Rset^n}f=0$ and $p_{n+1}=1$.  Then $F(P)=0$.  For $s\in [0,1]$, denote by $\tilde v_{\lambda, s}$ the  solution  to (with $f$ replaced by $sf$ in \eqref{ge5})
$$
\lambda \tilde v_{\lambda, s}+\left(1-d\,{\rm div}{\frac{p+D\tilde v_{\lambda, s}}{\sqrt{1+|p+D\tilde v_{\lambda, s}|^2}}}\right)\sqrt{1+|p+D\tilde v_{\lambda, s}|^2}+sf(x)=0 \quad \text{in $\Rset^n$}.
$$
Then by \eqref{no-cut-limit}, 
$$
\lim_{\lambda\to 0}(-\lambda \tilde v_{\lambda, s}(x))=\overline H_{s}(P,d) \quad \text{uniformly in $\Rset^n$}.
$$
Here the constant $\overline H_s(P, d)$ is continuous  with respect to the parameter $s$. 
Since $\overline H_0(P,d)=|P|=\sqrt{1+|p|^2}$, we can choose the largest  $t_0\in (0,1]$ such that
$$
\overline H_{t_0}(P,d) \geq 0.
$$
Let $\tilde v_0$ be the $C^{2,\alpha}$ solution to the cell problem
$$
\left(1-d\,{\rm div}{\frac{p+D\tilde v_{0}}{\sqrt{1+|p+D\tilde v_{0}|^2}}}\right)\sqrt{1+|p+D\tilde v_{0}|^2}+t_0f(x)=\overline H_{t_0}(P,d)\geq 0,
$$
which is then also a solution to the cell problem associated with the cutoff equation
$$
\left(1-d\,{\rm div}{\frac{p+D\tilde v_{0}}{\sqrt{1+|p+D\tilde v_{0}|^2}}}\right)_+\sqrt{1+|p+D\tilde v_{0}|^2}+t_0f(x)=\overline H_{t_0}(P,d)\geq 0.
$$
Let $v_{\lambda, t_0}$ be the solution to \eqref{ge6} with $f$ replaced by $t_0f$. Then the existence of cell problem immediately implies that
$$
\lim_{\lambda\to 0}(-\lambda v_{\lambda, t_0}(x))=\overline H_{t_0}(P,d)\geq 0 \quad \text{uniformly in $\Rset^n$}.
$$
There are two cases to be considered below.

\medskip

\noindent {\bf Case 1.} If $t_0=1$,  our conclusion holds with $\max\{\overline H(P,d), 0\}=\overline H(P,d)=\overline H_1(P,d)$. 

\medskip

\noindent {\bf Case 2.} If $t_0<1$, then $\overline H_{t_0}(P,d)=0$.  The comparison principle implies that $\overline H_1(P,d)\leq \overline H_{t_0}(P,d)=0$ and
$$
-\lambda v_{\lambda}(x)\leq -\lambda v_{\lambda, t_0}(x).
$$
Accordingly, 
$$
\limsup_{\lambda\to 0}(-\lambda v_{\lambda}(x))\leq \limsup_{\lambda\to 0}(-\lambda v_{\lambda, t_0}(x))=0.
$$
Combining with the assumption $\liminf_{\lambda\to 0}(-\lambda v_{\lambda}(x))\geq 0$,  we obtain that
$$
\lim_{\lambda\to 0}(-\lambda v_{\lambda}(x))=0=\max\{\overline H(P,d), 0\}\quad \text{uniformly in $\Rset^n$}.
$$
\end{proof}

The following statement follows immediately from the above lemma and the comparison principle. 

\begin{cor} \label{cor:order}
Suppose that $f_1$ and $f_2$ are two periodic Lipschitz continuous functions satisfying $p_{n+1}f_1\leq p_{n+1}f_2$ and $\max_{\Rset^n}(p_{n+1}f_1)= \max_{\Rset^n}(p_{n+1}f_2)$. 
Let $v_{\lambda, 1}$ and $v_{\lambda, 2}$ be solutions of \eqref{ge6} with $f$ being replaced by $f_1$ and $f_2$, respectively.   
If 
$$
\lim_{\lambda\to 0}(-\lambda v_{\lambda,1}(x))=\overline H_{+,1}(P,d)\quad \text{uniformly in $\Rset^n$}
$$
for some constant $\overline H_{+,1}(P,d)$, then 
$$
\lim_{\lambda\to 0}(-\lambda v_{\lambda,2}(x))=\overline H_{+,2}(P,d)\quad \text{uniformly in $\Rset^n$}.
$$
for some constant $\overline H_{+,2}(P,d)\leq \overline H_{+,1}(P,d)$.
\end{cor}

Now change $f$ to $Af$ for a parameter $A>0$ and write $ \overline H(P, d)$ and $\overline H_+(P, d)$ as $ \overline H(P, d,A)$ and $\overline H_+(P, d,A)$, respectively. Recall in addition the relation $\overline H_+(P, d,A) = \max \{\overline H(P, d,A),  \  A F(P) \}$.
Below are several  properties of $\overline H(P, d, A)$ with respect to the physical parameters $d$ and $A$. 
\begin{lem}\label{lem:prop}  The following properties hold.
\begin{itemize}
\item[(1)]
$$
|P|+A\min_{\Rset^n}(p_{n+1}f) \leq \overline H(P,d,A)\leq |P|+AF(P).
$$
\item[(2)]
$$
\lim_{A\to 0}\overline H(P,d,A)=|P|.
$$

\item[(3)]  For fixed $P \in \mathbb{R}^{n+1}$ and  $d>0$,  $\overline H(P,d,A)$ is differentiable with respect to $A$. Also,  if $p_{n+1}\not=0$ and $f$ is not a constant, then 
$$
{\partial \overline H(P,d,A)\over \partial A}<F(P).
$$
In particular,  this implies that  $\overline H(P,d,A)-F(P)A$ is strictly decreasing with respect to $A$ and $\overline H_+(P,d,A)-F(P)A$ is non-increasing with respect to $A$.

\item[(4)] For fixed $A$ and $P\in \Rset^{n+1}$, 
$$
\lim_{d\to 0}\overline H(P,d,A)=\overline H(P,0,A)\geq |p_{n+1}|+AF(P).
$$
Hence when $d$ is small enough, $ \overline H(P,d,A)=\overline H_+(P,d,A)$.

\item[(5)]  For fixed $A>0$ and $P=e_{n+1}=(0,0,\ldots,0,1)$, if $f$ is not a constant, then
$$
\overline H(e_{n+1},d, A)< 1+AF(e_{n+1})=\overline H(e_{n+1},0,A).
$$
\end{itemize}

\end{lem}

\begin{proof}  Change $f$ to $Af$ in equation \eqref{cell}.
Items (1) and (2) are obvious.  
Let us prove the last three items. 

\medskip

We first prove (3).   Due to boundedness of $|D\tilde v|$,  \eqref{cell} is uniformly elliptic. Hence, it has a unique $C^{2,\alpha}$ periodic solution $\tilde v(x,A)$ subject to $\tilde v(0,A)=0$. Thanks to the standard estimates for uniformly elliptic PDEs, $\tilde v(x,A)$ and $\overline H(P,d,A)$ are all differentiable with respect to the parameter $A$.  Differentiating \eqref{cell} with respect to $A$ leads to 
$$
-\sum_{i,j=1}^{n}a_{ij}(D\tilde v){\partial ^2v_{A}\over \partial x_i\partial x_j}+\sum_{i=1}^{n}b_{i}(D\tilde v,D^2\tilde v){\partial v_{A}\over \partial x_i}+p_{n+1}f(x)={\partial \overline H(P,d,A)\over \partial A}
$$
for $v_A(x)={\partial \tilde v(x,A)\over \partial A}$.  Note that ${\partial \overline H(P,d,A)\over \partial A}$ is actually the unique number such that the above linearized equation  (with fixed $\tilde v$) has a periodic solution.  Due to the gradient estimate in Lemma \ref{lem:estimate}, this is a uniformly elliptic equation satisfied by $v_A(x)$. 
We do not write down the explicit forms of $a_{ij}$ and $b_i$ as we do not really use them.
 If ${\partial \overline H(P,d,A)\over \partial A}\geq F(P)$, then 
$$
-\sum_{i,j=1}^{n}a_{ij}(D\tilde v){\partial ^2v_{A}\over \partial x_i\partial x_j}+\sum_{i=1}^{n}b_{i}(D\tilde v,D^2\tilde v){\partial v_{A}\over \partial x_i}\geq 0.
$$
Then $v_A$ has to be a constant by the strong maximum principle, which is absurd as $f$ is not constant.

\medskip

We next prove (4). 
Thanks to the gradient estimate in Lemma \ref{lem:estimate}, by passing to a subsequence if necessary as $d \to 0$,  $\tilde v(x,A)$ converges to a periodic Lipschitz continuous viscosity solution $\tilde v_0$ of 
$$
\sqrt{p_{n+1}^2+|p+D\tilde v_0|^2}+Ap_{n+1}f(x)=\overline H(P,0,A) \quad \text{$\Rset^n$}.
$$
The existence of $\overline H(P,0,A)$ is given in \cite{LPV}. 
\medskip

We finally prove (5).
Thanks to (1), $\overline H(e_{n+1},A,d)\leq 1+AF(e_{n+1})$.  
We just need to exclude the equality here.  
Assume by contradiction that $\overline H(e_{n+1},A,d)=1+A\max_{\Rset^n}f$. 
 Let $\tilde v$ be a $C^{2,\alpha}$ solution to  equation \eqref{cell}. Then
$$
\left(1-d\,{\rm div}{\frac{D\tilde v}{\sqrt{1+|D\tilde v|^2}}}\right)\sqrt{1+|D\tilde v|^2}\geq 1\quad \text{on $\Rset^n$}. 
$$
Therefore
$$
-\sum_{i,j=1}^{n}\tilde a_{ij}(D\tilde v){\partial ^2\tilde v\over \partial x_i\partial x_j}+\sum_{i=1}^{n}\tilde b_{i}(D\tilde v){\partial \tilde v\over \partial x_i}\geq 0.
$$
Again, we do not write down the explicit forms of $\tilde a_{ij}$ and $\tilde b_i$ as we do not really use them.
Then $\tilde v$ has to be a constant by the strong maximum principle, which gives a contradiction as $f$ is not constant.  
\end{proof}

Now denote by $v_\lambda(x,A)$  the unique solution to \eqref{ge6} with $f$ replaced by $Af$.  In general,  without assumption \eqref{c1},  it could happen that 
$$
-\lim_{\lambda\to 0}\lambda v_{\lambda}(x,A)=\overline H_+(P,d,A)>\overline H(P,d,A),
$$
which also leads to the instability of the set $S_H$.  Fix $n \ge2 $ and $d>0$,
we  choose a smooth periodic function  $\Psi(x)$ such that
$$
W(x)=\left(1-d\,{\rm div}{\frac{D\Psi}{\sqrt{1+|D\Psi|^2}}}(x)\right)\sqrt{1+|D\Psi|^2(x)}
$$
changes signs on $\Rset^n$.

\begin{lem}\label{lem:comparison} For $f(x)=-\max\{0, W(x)\}$,  then $F(e_{n+1})=0$ and 
\be\label{oneside}
-\lim_{\lambda\to 0}\lambda v_{\lambda}(x,A)=\overline H_+(e_{n+1},d,1)=0>\overline H(e_{n+1},d,1).
\ee
In particular, this implies that there exists $\tilde A<1$ such that 
\be\label{constant-value}
-\lim_{\lambda\to 0}\lambda v_{\lambda}(x,A)=\overline H_+(e_{n+1},d,A)=0=AF(e_{n+1}) \quad \text{ for $A\in [\tilde A, 1]$}.
\ee
In particular,  $\overline H_+(e_{n+1},d,A)-AF(e_{n+1}) $ is not strictly decreasing with respect to $A$. 
\end{lem}

\begin{proof} Since $\max_{\Rset^n}f=0$,  we have $F(e_{n+1})=0$.  The equality $-\lim_{\lambda\to 0}\lambda v_{\lambda}(x,A)=0=\overline H_+(e_{n+1},d,1)$ is obvious  since $\Psi$ is a smooth solution to the cell problem
$$
\left(1-d\,{\rm div}{\frac{D\Psi}{\sqrt{1+|D\Psi|^2}}}\right)_+\sqrt{1+|D\Psi|^2}+f(x)=0 \quad \text{in $\Rset^n$}.
$$
So $\overline H(e_{n+1},d,1)\leq 0$.  Note that
$$
\left(1-d\,{\rm div}{\frac{D\Psi}{\sqrt{1+|D\Psi|^2}}}\right)\sqrt{1+|D\Psi|^2}+f(x)\leq 0 \quad \text{in $\Rset^n$}.
$$
If  $\overline H(e_{n+1},d,1)=0$,  let $u$ be a periodic $C^{2,\alpha}$ solution to 
$$
\left(1-d\,{\rm div}{\frac{Dv}{\sqrt{1+|Dv|^2}}}\right)\sqrt{1+|Dv|^2}+f(x)=0 \quad \text{in $\Rset^n$}.
$$
Then $u-\Psi$  will be a smooth periodic supersolution to a uniformly elliptic equation on $\Rset^n$.   
Strong maximum implies that $u\equiv \Psi$, which contradicts to the fact that  $W(x)$ changes signs.  

 Let  $\tilde A$ be the smallest number such that $\overline H(e_{n+1},d,\tilde A)=0$.  
 Its existence  follows from (2) of Lemma \ref{lem:prop} and the continuity of $\overline H(e_{n+1},d,A)$ with respect to $A$.  
 Obviously, $\tilde A<1$ and $\overline H(e_{n+1},d,\tilde A)<0$ for $A\in (\tilde A, 1]$.  
 Then the equality \eqref{constant-value} follows immediately from  Corollary \ref{cor:order}. 
\end{proof}

\begin{rmk}\label{rmk:smooth} We would like to point out that \eqref{oneside} implies the instability of $S_H$. In fact, by suitable translation, we may assume $\Zset^2\subset\{x\in \Rset^2\,:\,  f(x)=\max_{\Rset^2}f=0\}$. Choose a sequence of smooth functions $\{f_m\}_{m\geq 1}$ satisfying that for $m\geq 1$, $\max_{\Rset^2}f_m=0$,  $\lim_{m\to \infty}\|f_m-f\|_{L^\infty(\Rset^n)}=0$ and
$$
\{x\in \Rset^2\,:\,  f_m(x)=0\}=\Zset^2.
$$
Let $S_H$ and $S_{H,m}$ be sets from Definition \ref{defin:set} corresponding to $f$ and $f_m$ respectively.   Then by item (3) in Theorem \ref{theo:main1},  we have that
$$
\lim_{m\to \infty}S_{H,m}=[0,\tilde A]\subsetneq [0,1]\subset S_H.
$$
In addition,   we can also find a smooth $\tilde f$ such that \eqref{oneside} holds. 
In fact,  let  $\{\tilde f_m\}_{m\geq 1}$ be a sequence of smooth functions satisfying that  for $m\geq 1$, $\max_{\Rset^n}\tilde f_m=0$, $\tilde f_m\geq f$ and $\lim_{m\to \infty}\|\tilde f_m-\tilde f\|_{L^\infty(\Rset^n)}=0$.  Let $v_{\lambda, m}$ be the solution to \eqref{ge6} with $f$ replaced by $\tilde  f_m$.  Owing to Corollary \ref{cor:order}, 
$$
- \lim_{\lambda\to 0}\lambda v_{\lambda,m}(x)=\overline H_{+,m}(e_{n+1},d,1) \quad \text{for  all $x\in \Rset^2$}.
$$
By the continuous dependence of both $\overline H$ and $\overline H_+$ on the function $f$,  when $m$ is large enough, we have that
$$
0=\overline H_{+,m}(e_{n+1},d,1)>\overline H_{m}(e_{n+1},d,1).
$$

\end{rmk}

\begin{rmk}\label{rmk:1d} 
When $n=1$ (two dimensional shear flows),    the limit $-\lim_{\lambda\to 0}\lambda v_{\lambda}(x)=\overline H_+(P,d,A)$ is easy to establish for all $P\in \Rset^2$ and $d, A\geq 0$.
The corresponding  equation \eqref{ge6} is reduced to an ODE
$$
\lambda v_{\lambda} +\left( \sqrt{p_{2}^2+|p_1+ v_{\lambda}^{'}|^2}-{\frac{dp_{2}^2\,  v_{\lambda}^{''}}{{p_{2}^2+|p_1+ v_{\lambda}^{'}|^2}}}\right)_++p_{2}f(x)=0.
$$
Applying the maximum principle to $v_{\lambda}^{'}(x)$ after suitable smooth approximations of $()_+$ if necessary,  we obtain an uniform bound
$$
|v_{\lambda}^{'}(x)|\leq C
$$
for a constant $	C$ independent of $\lambda \in (0,1)$. 
 It was proved in \cite{LXY} that $ \overline H(P,d,A)$ is strictly decreasing with respect to the Markstein number $d>0$ for non-constant $f$.  Thus $ \overline H_+(P,d,A)$ is non-increasing with respect to $d$.  In particular,  due to   (4) in Lemma \ref{lem:prop},  we have that   $\overline H_+(P,d,A)= \overline H(P,d,A)$ when $d$ is small enough.  Hence $\overline H_+(P,d,A)< \overline H_+(P,0,A)$ for all $d>0$. 
\end{rmk}

\section{Appendix: Full homogenization}

Throughout this appendix, we write $x'\in \Rset^n$ and $x=(x',x_{n+1})\in \Rset^{n+1}$.  
For given $\epsilon\in (0,1]$, let $G_{\epsilon}(x,t)
\in C(\Rset^{n+1}\times [0,\infty))$ be the unique viscosity solution to 
\be
\begin{cases}
{\partial G_\epsilon\over \partial t}  + \left(1-\epsilon d\,  {\rm div}\left({DG_\epsilon\over |DG_\epsilon|}\right)\right)_{+}|DG_\epsilon|+V({x\over \epsilon})\cdot DG_\epsilon=0,\label{ge1E}\\
G_{\epsilon}(x,0)=g(x).
\end{cases}
\ee
Here we assume that $g\in {\rm UC}(\Rset^{n+1})$.  
By using $2\arctan (G_{\epsilon})\over \pi$, we may also assume that 
$$
|G_\epsilon|,\ |g|\leq 1
$$
and thus, $g\in {\rm BUC}(\Rset^{n+1})$.   For $(S, P, x)\in S^{n+1}\times \Rset^{n+1}\times \Rset^{n+1}$, write
$$
F_\epsilon(S,P,x)= \left(|P|-d\epsilon\, \mathrm{tr} S+{d\epsilon\, P\cdot S\cdot P\over |P|^2}\right)_{+}+V\left({x\over \epsilon}\right)\cdot P
$$
Here $S^{n+1}$ is the set of $(n+1)\times (n+1)$ symmetric matrices. In this section,  we show how to get the homogenization of the curvature G-equation with shear flows based on the existence of correctors or approximate correctors. 
The proof uses the standard perturbed test function method together with the doubling variables method (\cite{CL1992,Evans}) due to possible lack of smooth correctors,  which is well-known to experts.  For the reader's convenience, we present it here. 

The following boundary control follows immediately from the comparison principle of the curvature G-equation (see  \cite[Theorem 2.1]{GLXY} for instance). 
\begin{lem}\label{lem:boundary}
For any $x_0\in \Rset^{n+1}$ and $\delta>0$, we have that
$$
|G_{\epsilon}(x,t)-g(x_0)|\leq 3\left(\left|{x-x_0\over \delta}\right|^2+l_{\delta,0}\right)^{1/2}+C_\delta t.
$$
Here 
\[
\begin{cases}
l_{\delta,0}=\max_{\bar B_\delta(x_0)}|g(x)-g(x_0)|,\\
C_\delta=\max_{x\in \Rset^{n+1}}\left((1+\max_{\Rset^{n+1}}|V|)|D\Psi(x)|+d(n+1)|D^2\Psi(x)|\right),\\
 \Psi(x)=3(\left|{x-x_0\over \delta}\right|^2+l_{\delta,0})^{1/2}.
\end{cases}
\]
\end{lem}

Now we assume the shear flow form $V(x)=(0,0,\ldots,0, f(x'))$ and $f:\Rset^n\to \Rset$ is periodic and Lipschitz continuous.  
Suppose that for all $P\in \Rset^{n+1}$, the following limit exists
\be \label{eq:appen}
\lim_{\lambda\to 0}\lambda v_{\lambda}(x')=-\overline H_+(P) \quad \text{for all $x'\in \Rset^n$}.
\ee
Here, $v_\lambda$ is the solution of \eqref{ge6} for each $\lambda>0$.
In particular, this implies the existence of approximate correctors: for any $\delta>0$, there exists  $v=v(x')\in C(\Rset^n)$ which is periodic and satisfies the following equation in viscosity sense in $\mathbb R^n$
\be\label{appro-cell}
\overline H_+(P)-\delta \leq \left(1-d\,  {\rm div}{\frac{p+Dv}{\sqrt{p_{n+1}^2+|p+Dv|^2}}}\right)_+\sqrt{|P+Dv|^2}+p_{n+1}f(x')\leq \overline H_+(P)+\delta .
\ee

Consequently,   the effective equation
\be\label{effective-eq}
\begin{cases}
G_t+\overline H_+(DG)=0\quad \text{in $\Rset^{n+1}\times (0,\infty)$},\\[3mm]
 G(x,0)=g(x).
\end{cases}
\ee
has a unique viscosity solution $G\in C(\Rset^{n+1}\times [0,\infty))$.

\begin{theo}\label{theo:hom}
Assume \eqref{eq:appen}.
Then,
$$
\lim_{\epsilon\to 0}G_{\epsilon}(x,t)=G(x,t)  \quad \text{locally uniformly on $\Rset^{n+1}\times [0, \infty)$}.
$$
Here, $G\in C(\Rset^{n+1}\times [0, \infty))$ is the unique viscosity solution to \eqref{effective-eq}. 

\end{theo}

\begin{proof}
 For $(x,t)\in \Rset^{n+1}\times (0,\infty)$, define
$$
\overline{G}(x,t)=\limsup_{\substack{\epsilon\to 0\\  y\to x,\, s\to t}}G_{\epsilon}(y,s)
$$
and
$$
{\underline G}(x,t)=\liminf_{\substack{\epsilon\to 0\\  y\to x,\, s\to t}}G_{\epsilon}(y,s).
$$
Clearly,  $\overline{G}$  is upper semicontinuous and $\underline G$ is lower semicontinuous and 
$$
-1\leq   \underline G\leq \overline{G}\leq 1.
$$

\medskip

\noindent {\bf Step 1.}  We show that $\overline{G}$ is a viscosity subsolution of the effective equation.  
Let us  argue by contradiction. 
If not, then there exists $(\bar x, \bar t)\in \Rset^{n+1}\times (0,\infty)$ and a smooth function $\phi(x,t)\in C^{\infty}(\Rset^{n+1}\times (0,\infty))$ such that 
$$
0=\overline{G}(\bar x, \bar t)-\phi(\bar x, \bar t)>\overline{G}(x,t)-\phi (x,t) \quad \text{for $(x,t)\not= (\bar x, \bar t)$}
$$
and
\be\label{positive1}
\phi_t(\bar x,\bar t)+\overline H_+(D\phi(\bar x,\bar t))=2\tau>0.
\ee
For $P=D\phi(\bar x, \bar t)=(p,p_{n+1})$,  let $v=v(x')\in C(\Rset^n)$ be a periodic continuous function satisfying equation \eqref{appro-cell} in the viscosity sense for $\delta={\tau\over 2}$.

{\it Case 1.} $p_{n+1}\not=0$.  For $\alpha>0$, set
$$
w(x,y,t)=w_1(x,t)-w_2(y,t)-{|x-y|^2\over 2\alpha}.
$$
Here $w_1(x,t)=G_\epsilon(x,t)-\phi(x,t)+P\cdot x-\overline H_+(P)t$ and $w_2(y,t)=P\cdot y-\overline H_+(P)t+\epsilon v({y'\over \epsilon})$. Note that 
$$
w(x,x, t)=G_{\epsilon}(x,t)-\phi(x,t)-\epsilon v\left({x'\over \epsilon}\right),
$$
which is the standard form of the perturbative test function from \cite{Evans}.  Apparently, for fixed $r>0$,  when $\epsilon$ and $\alpha$ are small enough,  $w$ attains a local maximum point at $(\hat x,\hat y,\hat t)\in B_r=B_r(\bar x)\times B_r(\bar x)\times B_r(\bar t)$
$$
w(\hat x,\hat y,\hat t)=\max_{\bar B_r} w(x,y,t).
$$
 Clearly, $\lim_{\epsilon,\delta\to 0}(\hat x,\hat y,\hat t)=(\bar x, \bar x, \bar t)$. We may choose $\epsilon$ and $\delta$ small enough such that 
$$
\phi(\hat x,\hat t)+\overline H_+(D\phi(\hat x,\hat t))\geq \tau={1\over 2}(\phi(\bar x,\bar t)+\overline H_+(P))>0.
$$

Owing to  (8.10) in the users' guide \cite{CL1992}, there are two $(n+1)\times (n+1)$ symmetric  matrices  $X$  and $Y$ and a real number $\hat a$ such that 
\be\label{control}
X\leq Y \quad \mathrm{and} \quad |q_1\cdot X\cdot q_1-q_2\cdot Y\cdot q_2|\leq {3\over \alpha}|q_1-q_2|^2
\ee
for all $q_1, q_2\in \Rset^{n+1}$ and
$$
(\hat a, \hat q,  X)\in { \mathcal {\bar P}}_{\mathcal {O}}^{2,+}w_1(\hat x, \hat t) \quad  \mathrm{and} \quad (\hat a, \hat q,  Y)\in { \mathcal {\bar P}}_{\mathcal {O}}^{2,-}w_2(\hat x, \hat t) 
$$
for
$$
\hat q={\hat x-\hat y\over \alpha}. 
$$
Then
$$
(\hat a+\phi_t(\hat x,\hat t)+\overline H(P), \ \hat q+D\phi(\hat x, \hat t)-P, \ X+D^2\phi(\hat x,\hat t))\in  { \mathcal {\bar P}}_{\mathcal {O}}^{2,+}G_{\epsilon}(\hat x, \hat t).
$$
See Section 8 in  \cite{CL1992} for the definition of corresponding notations. 
Then $\hat q\cdot e_{n+1}=p_{n+1}$. 
 In particular, $|\hat q|\geq |p_{n+1}|=\theta>0$. 
 Hereafter $C$ represents a constant that depending only on $\phi$, $u$ and $\theta$.  
 Moreover, using
$w(\hat x, \hat x, \hat t)\leq w(\hat x, \hat y, \hat t)$, we derive that
$$
{|\hat x-\hat y|^2\over 2\alpha}\leq |w_2(\hat x, \hat t)-w_2(\hat y, \hat t)|\leq C(|\hat x-\hat y|+\epsilon).
$$
Accordingly, 
\be\label{q-bound}
|\hat q|\leq C\left (1+\sqrt{\epsilon\over \alpha}\right).
\ee

 Note that for the function $S(v)={v\over |v|}$, 
\be\label{gradient-control}
|DS(v)|\leq {1\over |v|}
\ee
Then when $\epsilon$ and $\alpha$ are small enough,
\be\label{gradient-difference}
|S(\hat q)-S(\hat q+D\phi(\hat x, \hat t)-P)|\leq {C\hat E\over |\hat q|+\theta}, 
\ee
where $\hat E=|D\phi(\hat x, \hat t)-P|$.  Obviously, $\lim_{\epsilon\to 0, \alpha\to 0}\hat E=0$.   Plugging into  equation \eqref{ge1E} in viscosity sense gives
$$
\hat a+\phi_t(\hat x,\hat t)+\overline H(P)+F_\epsilon(X+D^2\phi(\hat x,\hat t), \hat q+D\phi(\hat x, \hat t)-P, {\hat x})\leq 0
$$
and by \eqref{appro-cell} in viscosity sense, 
$$
\hat a+F_\epsilon(Y, \hat q, {\hat y})\geq -{\tau\over 2}.
$$
Accordingly, 
$$
\hat a+F_\epsilon(X+D^2\phi(\hat x,\hat t), \hat q+D\phi(\hat x, \hat t)-P, {\hat x})\leq -{\tau}
$$
Taking the difference of the above two equations,  thanks to \eqref{control},  \eqref{gradient-control} and \eqref{gradient-difference}, we derive that
$$
C\left(\epsilon+{3\epsilon\over \alpha}{\hat E^2\over |\hat q|^2+\theta^2}+{\alpha |\hat q|\over \epsilon}\right)\geq {\tau\over 2}.
$$
Then
$$
C\left(\epsilon+{3\epsilon\over \alpha}{\hat E^2\over \theta^2}+{\alpha |\hat q|\over \epsilon}\right)\geq {\tau\over 2}
$$
Combining with \eqref{q-bound}, 
$$
C\left(\epsilon+{3\epsilon\over \alpha}{\hat E^2\over \theta^2}+{\alpha\over \epsilon}+\sqrt{{\alpha\over \epsilon}}\right)\geq {\tau\over 2}.
$$
Fix $K\in \Nset$ and let $\alpha={\epsilon\over K}$. Then
$$
C\left(\epsilon+{3K\hat E^2\over \theta^2}+{1\over K}+\sqrt{{1\over K}}\right)\geq {\tau\over 2}.
$$

 Sending $\epsilon\to 0$ leads to 
$$
C\left({1\over K}+\sqrt{{1\over K}}\right)\geq {\tau\over 2}.
$$
which is impossible when $K$ is large enough. 

\medskip

{\it Case 2.} $p_{n+1}=0$.  Then $\overline H_+(P)=|p|$ and $v=0$.  For fixed $r$, choose $\epsilon$ small enough, such that
$$
G_\epsilon(x,t)-\phi(x,t)
$$
attains local maximum at some point $(\hat x, \hat t)\in B_r(\bar x,\bar t)$.  Plugging into equation \eqref{ge1E} in viscosity sense, we have that
$$
\phi_t(\hat x,\hat t)+|p|-C(\epsilon+|\hat x-\bar x|+|\hat t-\bar t|)\leq 0.
$$
Sending $\epsilon$, $\delta\to 0$ leads to 
$$
\phi_t(\bar x, \bar t)+\overline H_+(P)=\phi_t(\bar x, \bar t)+|p|\leq 0,
$$
which contradicts \eqref{positive1}. 

\medskip

\noindent {\bf Step 2.}  Similarly, we can show that $\underline G$ is a viscosity supersolution of the effective equation.  Hence by the comparison principle
$$
\underline G\geq \overline{G}.
$$
Therefore,  $\underline G=\overline{G}=G$.  
Besides, by Lemma \ref{lem:boundary}, 
$$
\lim_{t\to 0}G(x,t)=g(x) \quad \text{locally uniformly on $\Rset^{n+1}$}.
$$
Therefore, $G$ is the unique solution of the effective equation \eqref{effective-eq}.
\end{proof}

\section{Acknowledgement}
We would like to thank the anonymous referee for carefully reading our paper and providing  very helpful suggestions to improve its presentation.

\section*{Data availability}
Data sharing not applicable to this article as no datasets were generated or analyzed during the current study.

\section*{Conflict of interest}
There is no conflict of interest.

\bibliographystyle{plain}

\end{document}